\documentclass[11pt]{amsart}
\usepackage{amssymb,amsmath,amsfonts,epsfig,graphics}
\theoremstyle{plain}

\numberwithin{equation}{section}

\def\la{\lambda}

\def\pa{\partial}
\def\f{\frac}
\def\na{\nabla}

\def\t{\tilde}

\def\g{\gamma}

\newtheorem{thm}{Theorem}[section]
\newtheorem{lem}{Lemma}[section]

\theoremstyle{remark}
\newtheorem{rem}{Remark}
\numberwithin{equation}{section}
\newcommand \RR   {\mathbb{R}}
\def\be{\begin{equation}}
\def\ee{\end{equation}}

\def\la{\label}

\def\pa{\partial}
\def\f{\frac}
\def\na{\nabla}

\def\t{\tilde}

\def\g{\gamma}

\def\r{\rho}

\def\f{\frac}

\def\t{\tau}

\def\be{\begin{equation}}
\def\ee{\end{equation}}
\def\la{\label}

\begin{document}
\title{Existence of Magnetic  Compressible Fluid Stars}
\author{Paul Federbush, Tao Luo, Joel $\text {Smoller} ^1$}

\date{}
\maketitle

\footnotetext[1]{Smoller was supported by the NSF, Grant no, DMS-1105189.}
\begin{abstract}
The existence of magnetic star solutions which are axi-symmetric stationary  solutions  for the  Euler-Poisson system of compressible fluids coupled to a magnetic field is proved in this paper by a variational method. Our method of proof consists of deriving an elliptic equation for the magnetic potential  in cylindrical coordinates in $\mathbb{R}^3$, and obtaining the estimates of the Green's function for this elliptic equation by  transforming it to 5-Laplacian. \end{abstract}

\section{Introduction}
The purpose of this paper is to prove the existence of  magnetic star solutions
which are axi-symmetric stationary  solutions  for the following  Euler-Poisson system of compressible fluids coupled to a magnetic field 
 (cf. \cite{11, 57, 68}):
\begin{equation}\label{MHD}
\begin{cases}
&\rho_t+\nabla\cdot(\rho {\bf v})=0,\\
&(\rho {{\bf v}})_t+\nabla\cdot(\rho {\bf v}\otimes {\bf v})+\nabla p(\rho)=-\rho \nabla \Psi+\frac{1}{4\pi}(\nabla\times {\bf B})\times {\bf B},\\
&\f{\pa {\bf B}}{\pa t}=\na\times ({\bf v}\times {\bf B}),\\
& \nabla\cdot {\bf B}=0, \\
&\Delta \Psi=4\pi G\rho.
\end{cases}
\end{equation}
Here $\rho$, ${\bf v}=(v_1, v_2, v_3)$, ${\bf B}$,  $p(\rho)$ and $\Phi$ denote
the density, velocity, magnetic field, pressure and gravitational
potential, respectively.  $G$ is the gravitational constant;  we set it equal to 1 for
simplicity.
  The gravitational potential is given by
\be\label{phi}\Phi(x)=-\int_{\RR^3} \frac{\rho(y)}{|x-y|}dy =-\rho\ast \frac{1}{|x|},\ee where $\ast$ denotes convolution.

Throughout this paper, we assume that the pressure function $p(\rho)$ satisfies the usual $\g$-law,
\be p(\rho)=\rho^{\g},\  \rho\ge 0, \ee
for some  constant $\g>1$.

In this paper, we are interested in the stationary axi-symmetric solutions of (\ref{MHD})  which represent an important class of equilibrium configurations. The stationary solutions  ( ${\bf v}={\bf 0}$) satisfy the following system:

\begin{equation}\label{1}
\begin{cases}
&\nabla p(\rho)=-\rho \nabla \Phi+\frac{1}{4\pi}(\nabla\times {\bf B})\times {\bf B},\\
& \nabla\cdot {\bf B}=0, \\
&\Delta \Phi=4\pi \rho.
\end{cases}
\end{equation}

There have been extensive studies on gaseous stars without taking  magnetic effects into account, both for non-rotating and rotating stars;  the reader may refer to \cite{ch, AB, Au, CF, FT1, FT2, Li1, Li2, LS, 38, mc} for the existence and properties of those solutions, and to
\cite{lin, lebovitz, lebovitz1, 44, 45, Rein, guo1, Jang, janginstability, 90} for stability and instability (nonlinear or linear) in various settings.  However, as far as we know,  there have been no rigorous mathematical results on  magnetic stars. The effects of magnetic fields arise in some physically interesting and important phenomena in astrophysics; eg.~solar flares. As noted in \cite{11}: ``{\it The coupling between magnetic and thermomechanical degrees of freedom is observed in the solar flares  (eruption phenomena
in the coronal region of the Sun). During this spectacular event, a violent brightening
is produced in the solar atmosphere where a huge amount of energy ($\sim\  10^{25}$ joules) is
released in a matter of few minutes, and associated to a large coronal mass ejection.
Magnetic reconnection is thought to be the mechanism responsible for this conversion of magnetic energy into heat and fluid motion. }'' The aim of this paper is to give the first proof of the existence of stationary magnetic star solutions with  prescribed total mass.

We prove our existence theorem via a variational technique as done in the non-magnetic case; see eg.~\cite{AB, Li2, FT2, 44, 45, Rein}. In these papers, the idea is to minimize an energy functional over a certain class $W_{M}$ of $\rho \in L^{\gamma}(\mathbb{R}^{3})$, $\gamma \geq 4/3$, subject to a total mass constraint
\[
	\int_{\mathbb{R}^{3}}\rho(x)dx = M,
\]
where $M$ is a given positive number. The principal mathematical difficulty is that the energy functional is not of fixed sign. In our case where the Euler-Poisson equations are coupled to a magnetic field, the problem becomes more challenging.

The coupling to a magnetic field alone arises because stars seldom have a net charge (\cite{lin}), so the electric field vanishes. Thus in order to take electro-magnetic effects into account, a non-trivial current ${\textbf{J}}$ must be present. Since the current in a star is quite complicated (and not known even for the Sun), we take a special ansatz for ${\textbf{J}}$. Proving the existence of a solution to the equations (\ref{1}) with this ansatz demonstrates the consistency of our model.

There are two important inequalities needed in our existence results; namely, if $F(\rho)$ denotes the energy functional defined on some class of functions $W_{M}$, then we need to prove
\begin{equation}
	\inf_{\rho \in W_{M}} F(\rho) < 0, \tag{i}
\end{equation}
and
\begin{equation}
	-\infty < \inf_{\rho \in W_{M}} F(\rho). \tag{ii}
\end{equation}
Inequality (i) shows that the gravitational energy dominates the other terms in the energy functional so that the star ``holds together." The second inequality implies that on any minimizing sequence, the energy functional is bounded from below.

In $\S$2 we set up the problem in a convenient manner. In $\S$3 we frame the problem variationally. In $\S$4 we prove the main theorem. This states that if $\gamma > 2$, the energy functional has a minimizer in $W_{M}$.

Our method of proof consists of deriving an elliptic equation for the magnetic potential $\psi$ in cylindrical coordinates, $r = (x_{1}^{2} + x_{2}^{2})^{1/2}$, $z = x_{3}$; namely
\[
	\psi_{rr} - \frac{1}{r}\psi_{r} + \psi_{zz} = -4\pi\beta r^{2}\rho
\]
where $\rho \in W_{M}$ and $\beta$ is a free parameter. To solve this equation we first transform $\psi(x)$ to a certain function $\chi(x)$, $x \in \mathbb{R}^{3}$. Then we extend $\chi$ and $\rho$ to functions $\chi_{e}$ and $\rho_{e}$ on $\mathbb{R}^{5}$ which satisfy
\[
	\Delta_{5}\chi_{e} = -4\pi\beta\rho_{e},
\]
where $\Delta_{5}$ denotes the 5-Laplacian. This enables us to write $\chi_{e}$ as a convolution of $\rho_{e}$ with the Green's function for $\Delta_{5}$. Using H$\ddot{\text{o}}$lder's inequality together with Young's inequality, we can estimate $\chi_{e}$ and thus $\psi$ too. These estimates are used to prove (i) and (ii) if $\gamma > 2$.

In the appendix we extend our results to $\gamma = 2$ for sufficiently small $\beta$. This result seems relevant for computing a ``Chandrasekhar (mass) limit" of certain recently discovered white dwarf stars, cf.~\cite{das}. By further restricting the class $W_{M}$, we employ Riesz potentials (\cite{gilbargtrudinger}), to extend our results to $\gamma > 8/5$. In a second appendix, we prove the non-existence of stationary spherically symmetric magnetic stars; cf.~\cite{38}.

\section{Formulation of the Problem}

We consider axi-symmetric solutions of (\ref{1}).
Thus if $x=(x_1, x_2, x_3)\in \RR^3$, let
$r=\sqrt {x_1^2+x_2^2},\ z=x_3$. The solutions we seek have the form
\be\label{hh1}\begin{cases}
& \rho(x)= \rho(r, z)), \Phi(x)= \Phi(r, z),\\
& {\bf B}(x)= B^r(r, z){\bf e}_r+ B^{\theta}(r, z){\bf e}_{\theta}+B^z(r, z){\bf e}_3.
 \end{cases}\ee
  Here
  \be {\bf e}_r=(x_1/r, x_2/r,  0)^\mathrm{T},\ {\bf e}_{\theta}=(-x_2/r,  x_1/r,\ 0)^\mathrm{T},\ {\bf e}_3=(0, 0, 1)^\mathrm{T},\ee
  so $\{{\bf e}_r, \ {\bf e}_{\theta},\ {\bf e}_3\}$ is the standard orthogonal  basis in cylindrical coordinates.  In this case,
   \be\label{2}
  {\bf B}=(\frac{x_1}{r}B^r-\frac{x_2}{r}B^{\theta}){\bf i}+(\frac{x_2}{r}B^r+\frac{x_1}{r}B^{\theta}){\bf j}+B^z{\bf k},
  \ee
 and thus
   \be\label{curltoday}
   \nabla\times {\bf B}=(\frac{x_2}{r}g-\frac{x_1}{r}\pa_z B^{\theta}){\bf i}
                                         -(\frac{x_1}{r}g+\frac{x_2}{r}\pa_z B^{\theta}){\bf j}
                                         +(\frac{1}{r}B^{\theta}+\pa_r B^{\theta}){\bf k},\ee
where
   \be\label{eq2.5}g=(\pa_r B^z-\pa_z B^r).\ee
   Furthermore,
    \begin{align}\label{curl1}
   &(\nabla\times {\bf B})\times {\bf B}\notag\\
   &=\left[-\frac{x_1}{r}\left(gB^z+\frac{(B^{\theta})^2}{r}+B^{\theta}\pa_rB^{\theta}\right)
     -\frac{x_2}{r}\left(B^z\pa_zB^{\theta}+\frac{B^rB^{\theta}}{r}+B^r\pa_r B^{\theta}\right)\right]{\bf i}\notag\\
     &+\left[-\frac{x_2}{r}\left(gB^z+\frac{(B^{\theta})^2}{r}+B^{\theta}\pa_rB^{\theta}\right)+\frac{x_1}{r}\left(B^z\pa_zB^{\theta}+\frac{B^rB^{\theta}}{r}+B^r\pa_r B^{\theta}\right)\right]{\bf j}\\
  &+(gB^r-B^{\theta}\pa_z B^{\theta}){\bf k}.\notag \end{align}
Also, it is easy to show
  \be\label{3}
   \nabla p(\rho)+\rho \nabla \Phi
   =\frac{x_1}{r}(\pa_r p(\rho)+\rho\pa_r\Phi){\bf i}
   +\frac{x_2}{r}(\pa_r p(\rho)+\rho\pa_r\Phi){\bf j}
   +(\pa_z p(\rho)+\rho\pa_z\Phi){\bf k}.
   \ee
   Thus (\ref{1}), (\ref{curl1}) and (\ref{3}) imply that
   $$B^z\pa_zB^{\theta}+\frac{B^rB^{\theta}}{r}+B^r\pa_r B^{\theta}=0.$$
   If $B^{\theta}=0$, then this is clearly satisfied.  For simplicity, we consider the case \be\label{btheta} B^{\theta}=0.\ee
   With this assumption, (\ref{curltoday}) reduces to
   $$\label{curl2}
   \nabla\times {\bf B}=g(\frac{x_2}{r}{\bf i}-\frac{x_1}{r}{\bf j}),$$
     where $g$ is given in \eqref{eq2.5}.
     The magnetic current ${\bf J}$ is defined by
     $$\nabla\times {\bf B}=\frac{4\pi}{c}{\bf J}. $$
     If $f=\frac{cg}{4\pi r}$, where $c$ is the speed of the light,
our ansatz for the current density ${\bf J}$ is the simplest one that supports a magnetic field; namely
   \be\label{currentansatz} {\bf J}=(x_2{\bf i}-x_1{\bf j})f(r, z). \ee
   Conversely, we can show that if the current density ${\bf J}$ takes the form of \eqref{currentansatz}, then $B^{\theta}=0.$. Indeed, with ${\bf B}$ given in (\ref{hh1}), we have:
\begin{align}\label{curl} &\frac{4\pi}{c}{\bf J}=\nabla\times {\bf B}\notag\\
&= \left(\frac{x_2}{r}(\pa_rB^z-\pa_zB^r)-\frac{x_1}{r}\pa_z B^{\theta}\right){\bf i}\\
  &+\left(\frac{x_1}{r}(\pa_zB^r-\pa_rB^z)-\frac{x_2}{r}\pa_z B^{\theta}\right){\bf j}\notag\\
  &+\left(\frac{1}{r}B^{\theta}+\pa_r B^{\theta}\right){\bf k}.\notag \end{align}
 Therefore, if  ${\bf J}$ takes the form of \eqref{currentansatz}, we have:
   \be\label{btheta2} B^{\theta}(r, z)=0,\ee
   and in this case,
   \be\label{bb} \frac{1}{r}(\pa_r B^z-\pa_z B^r)=\frac{4\pi}{c}f(r, z).\ee
   Next $\nabla\cdot {\bf B}=0$ implies:
   \be\label{divergence} \pa_r B^r+\frac{1}{r}B^r+\pa_z B^z=0.\ee
An easy calculation gives
   \be
   (\nabla\times {\bf B})\times {\bf B}=-\frac{x_1}{r}B^z(\pa_r B^z-\pa_z B^r){\bf i}
                                         -\frac{x_2}{r}B^z(\pa_r B^z-\pa_z B^r){\bf j}
                                         +B^r((\pa_r B^z-\pa_z B^r){\bf k}.\ee
This together with (\ref{bb}) implies
   \be\label{baby} \frac{1}{4\pi}(\nabla\times {\bf B})\times {\bf B}=\frac{f}{c}\left(-x_1B^z{\bf i}
   -x_2B^z{\bf j}+r B^r{\bf k}\right).\ee
  Therefore, we have, by \eqref{1}, \eqref{3} and \eqref{baby},
   \be\label{maineq}\begin{cases}
   &\pa_r p(\rho)+\rho\pa_r\Phi=-\frac{r f}{c}B^z,\\
   &\pa_z p(\rho)+\rho\pa_z\Phi=\frac{r f}{c}B^r.
   \end{cases}\ee

   Let \be\label{irho} i(\rho)=\int_0^{\rho} \frac{p'(s)}{s}ds.\ee Then \eqref{maineq} implies
   \be\label{maineq1}\begin{cases}
   &\rho \pa_r (i(\rho)+\Phi)=-\frac{r f}{c}B^z,\notag\\
   &\rho \pa_z (i(\rho)+\Phi)=\frac{ r f}{c}B^r.
   \end{cases}\ee
   Now writing \eqref{divergence} in the form
   \be\label{divergence1} \pa_r (r B^r)+\pa_z (r B^z)=0\ee
   enables us to introduce a magnetic potential $\psi$ such that
   \be\la{magpotential}
   \pa_z \psi =rB^r,\  \pa_r\psi=-rB^z.\ee

   In this paper, we consider the case when
   \be\label{frhobeta}\frac{f}{c\rho}=const=:\beta. \ee
Then it follows from \eqref{maineq1} and \eqref{magpotential} that
$$\nabla (i(\rho)+\Phi-\beta\psi)=0, \ {\rm whenever~} \rho>0. $$
Hence,
  \be\la{zero1} i(\rho)+\Phi-\beta\psi=const=:\lambda, \ {\rm in~the~region~} \rho>0,\ee
  where $i(\rho)$ is given by \eqref{irho}, and
  $\Phi$ is given by
  \be\la{grho}\Phi(x)=-\int_{\mathbb{R}^3} \frac{\rho(y)}{|x-y|}dy=:-\mathfrak{G}(\rho)(x).\ee
  Then solving \eqref{zero1} with the total mass constraint
\be\label{massconstraint}
\int_{\mathbb{R}^3} \rho(x)dx=M, \ {\rm~ for~some ~given~positive~constant~} M,
\ee
is the problem we consider in this paper.

\section{Variational formulation}
 For $p$ satisfying the
  $\gamma$-law, (1.3), let
  \be\label{arho}A(\rho)=\frac{p(\rho)}{\gamma-1}.\ee
  Then
  \be\label{ai}i(\rho)=A'(\rho).\ee
  Also, the gravitational potential is given by \eqref{grho}, and we write $\Phi=-\mathfrak{G}(\rho)$.
  The magnetic potential $\psi$ satisfies
  \be\la{psi} {\rm div} (\frac{1}{r^2}\nabla \psi)=-4\pi\beta \rho,\ee
  where $r=\sqrt {x_1^2+x_2^2}.$
  Let $G(x, y)$ for $x, y\in \mathbb{R}^3$ be the Green's function for the operator ${\rm div} (\frac{1}{r^2}\nabla)$, i. e.,
  \be\la{green} LG=: {\rm div} (\frac{1}{r^2}\nabla G(x, y))=\delta(x-y)),\ee
  where $\delta(x-y)$ is the Dirac mass.
  Since $L$ is symmetric, we have
  \be\la{ajoint} <L\psi, G>=<\psi, LG>=<\psi, \delta(x-y)>=\psi(y), \ee
  where the inner product $<\cdot, \cdot>$ is taken in $L^2$. Thus we have the following integral representation for $\psi$, namely,
  \be\la{integral} \psi(x)= \mathfrak{P}(\rho), \ee
  where the integral operator $\mathfrak{P}$ is given by
  \be\la{mathfrakpsi}\mathfrak{P}(\rho)=-4\pi\beta \int_{\mathbb{R}^3} G(x, y) \rho(y)dy.\ee
  Then, equation \eqref{zero1} can be written as
  \be\la{zero} i(\rho)-\mathfrak{G}(\rho)-\beta\mathfrak{P}(\rho)=\lambda, \ {\rm whenever~} \rho>0.\ee
\vskip 0.3cm
  In order to state our results, let's review the following results for the non-rotating non-magnetic star solutions:  For $0<M<+\infty$, define $X_M$ by
\begin{align}\label{5.2} X_M&=\{\r: \RR^3\to \RR, \rho\ge 0, a.e.,\
 \int_{\mathbb{R}^3}\rho(x)dx=M, {\rm and}\notag\\
&\int_{\mathbb{R}^3}
  [A(\rho(x))+\frac{1}{2}\rho(x) \mathfrak{G}(\rho)(x)]dx<+\infty \},  \end{align}
  where $A(\r)$ is the function given in \eqref{arho}.
 For $\rho\in X_M$,  we define the {\bf energy functional} $\tilde F$ for non-rotating non-magnetic  stars by
  \begin{align}\label{55.3}
  \tilde F(\rho)=\int
  [A(\rho(x))-\frac{1}{2}\rho(x) \mathfrak{G}(\rho(x))]dx.
  \end{align}
 We then have
 \begin{thm}\label{5.1'} Suppose that the pressure function $p(\rho)=\rho^{\gamma}$ with $\gamma>4/3$.
Let $\hat \rho$ be a minimizer of the energy functional $\tilde F$ in $X_M$ and let
\be\label{G1'}
\Gamma_M=\{x\in \RR^3:\  \hat \rho (x)>0\}, \ee then
 there exists a constant $\lambda_N$
such that
\be\label{lambda111}
\begin{cases}
& A'(\hat \rho(x))-B\hat \rho(x)=\lambda_M, \qquad x\in \Gamma_M,\\
&-\mathfrak{G}(\hat \rho)(x)\ge \lambda_N, \qquad x\in \RR^3-\Gamma_M.\end{cases}\ee
\end{thm}
\noindent The proof of this theorem is well-known, cf.  \cite{AB} or \cite{Rein}.
\vskip 0.2 cm
\begin{rem} We call the minimizer $\hat \rho$  of the functional $\tilde F$ in $X_M$  a {\it non-rotating non-magnetic star solution}.
\end{rem}
\vskip 0.2cm
\noindent
\begin{rem} For $\gamma>4/3$, it was proved in  \cite{38} that such a minimizer  $\hat \rho$  of the functional $\tilde F$ in $X_M$ exits and  is actually radial and unique, and has  compact support, i. e., for the given total mass $M$, there exists a unique constant $R_M>0$ such that
\be\la{radius}
\begin{cases}
&\hat\rho(x)>0, \ {\rm if~} |x|<R_M, \\
&\hat\rho(x)=0, \ {\rm if~} |x|\ge R_M.
\end{cases}
\ee
In this case, we call $R_M$ the {\it radius} of the  non-rotating non-magnetic star solution with prescribed total mass $M$ $^3$.\end{rem} \footnotetext[3]{In the appendix, we prove that radial solutions do not exist when magnetic fields are present.}

  Let $W_M$ be the following function space
  \begin{align*}W_M=&\{\r: \RR^3\to \RR,\ \r  {\rm~is~axisymmetric, ~}\rho\ge 0, a.e.,\ \rho\in L^1(\RR^3)\cap L^{\gamma}(\RR^3),\\
 &\int\rho(x)dx=M\},\end{align*}
 and let $W_M^{\ast}$ be defined by
 \be\label{key}
 W_M^{\ast}:=\{\rho\in W_M: \rho(r, z)=0 \ {\rm for~} r\ge \mathfrak{R}\},
\ee
for some positive constant  $\mathfrak{R}\ge R_M$ where $R_M$ is the radius of the non-rotating non-magnetic star solution with prescribed total mass $M$, given in \eqref{radius} .
\vskip 0.2cm

  Define a functional $F$ on  $W_M^{\ast}$ by
  \be\la{functional} F(\rho)=\int \left (A(\rho)-\frac{1}{2}\rho\mathfrak{G}(\rho)-\frac{1}{2}\beta\rho\mathfrak{P}(\rho)\right)dx.\ee
  We now show that a minimizer of the functional $F$ in $W_M^{\ast}$ solves equation \eqref{zero}.
  \begin{thm}\label{criticalpt} Let $\tilde \r$ be a minimizer of the energy functional $F$ in $W_M^{\ast}$ and let \be\label{G1}
\Gamma_M=\{x\in \RR^3:\  \tilde \r(x)>0\}.\ee If $\g>6/5$, then $\tilde \r\in C(\RR^3)\cap C^1(\Gamma)$.
Moreover,
 there exists a constant $\lambda_M$
such that
\be\label{lambda1}
 A'(\tilde \r(x))-\mathfrak{G}(\tilde \r)(x)-\beta \mathfrak{P}(\tilde \r)(x)=\lambda_M, \qquad x\in \Gamma.\ee
\end{thm}
\begin{proof}
 We write $F(\rho)$ in two parts:
 \be\la{f} F(\rho)=\tilde F(\rho)+I_2(\rho),\ee
 where
 $$\tilde F(\rho)=\int_{\mathbb{R}^3} \left (A(\rho)-\frac{1}{2}\rho\mathfrak{G}(\rho)\right)dx,$$
 and
 $$I_2(\rho)=-\frac{1}{2}\beta\int_{\mathbb{R}^3}\rho\mathfrak{P}(\r) dx.$$
 For and $\rho\in W_M^{\ast}$ and $\rho+t\sigma\in W_M^{\ast}$ for $t\in \mathbb{R}$ and $\int_{\mathbb{R}^3}\sigma dx=0$, then using the same argument as in \cite{AB}, we have
 \be\la{g}\lim_{t\to 0} \frac{\tilde F(\rho+t\sigma)-\tilde F(\rho)}{t}=\int (i(\rho)-\mathfrak{G}(\r))\sigma dx.\ee
 We calculate $I_2(\rho+t\sigma)-I_2(\rho)$ as follows:  by using \eqref{integral} and (3.7), we obtain
\begin{align}
 & I_2(\rho+t\sigma)-I_2(\rho) \notag\\
 &=2\pi \beta^2\int\int G(x, y)\{(\rho+\t\sigma)(x)(\rho+t\sigma)(y)-\rho(x)\rho(y)\}dxdy\notag\\
 &=2\pi t \beta^2\int\int G(x, y)(\sigma(x)\rho(y)+\rho(x)\sigma(y))dxdy+2\pi\beta^2 t^2\int\int G(x, y)\sigma(x)\sigma(y)dxdy.
 \end{align}
Since $G(x, y)=G(y,x)$, we thus have
 \be\la{I22}\lim_{t\to 0} \frac{I_2(\rho+t\sigma)-I_(\rho)}{t}=4\pi \beta^2 \int\int G(x, y)(\rho(y)\sigma(x)dxdy=-\beta\int \mathfrak{P}(\r)(x)\sigma(x)dx.\ee
 Therefore, by \eqref{g} and \eqref{I22}, we get
 \be\la{ff} \lim_{t\to 0} \frac{F(\rho+t\sigma)-F(\rho)}{t}=\int (i(\rho)-\mathfrak{G}(\r)-\beta \mathfrak{P}(\r))\sigma (x)dx,\ee
 for all $\sigma$ such that $\int\sigma(x)dx=0$. This, together with \eqref{g} proves the theorem, using a similar argument as in \cite{AB}.\end{proof}
\vskip 0.2cm
 The main theorem of this paper is the following:
 \vskip 0.2cm
 \begin{thm}\label{maintheorem} Suppose  that $\gamma>2$.   Then
the following three statements hold:\\
(1)
\be\label{ne}\inf_{W_M^{\ast}} F(\rho)<0, \ee
and
\be\label{boundedbelow}
F(\rho)\ge C_1\int_{\mathbb{R}^3}\rho^{\gamma}d_3x-C_2, \qquad \rho\in W_M^{\ast},
\ee
for some positive constants constants $C_1$ and $C_2$ independent of $\rho$.
\vskip 0.2cm
\noindent (2) if $\{\r^i\}\subset W_M^{\ast} $ is a minimizing sequence for the functional $F$,
then there exists a sequence of vertical shifts $a_i{\bf e_3}$ ($a_i\in \RR$, ${\bf e_3}=(0, 0, 1)$),  a subsequence of $\{\r^i\}$,  (still labeled  $\{\r^i\}$),  and a function $\tilde \r\in W_M^{\ast}$, such that for any $\epsilon>0$ there exists $R>0$ with
\be\label{2.15}
\int_{a_i{\bf e_3}+B_R(0)}\r^i(x)dx\ge M-\epsilon, \quad i\in \mathbb{N},\ee
and
\be\label{2.16} T\r^i(x):=\r^i(x+a_i{\bf e_3})\rightharpoonup \tilde \r,\  weakly~in~L^{\g}(\RR^3),\ as\  i\to \infty.\ee

\vskip 0.2cm
\noindent Moreover\\
(3) $\tilde \r$ is a minimizer of $F$ in $W_M^{*}$.

\end{thm}

\vskip 0.4cm
Notice that (3.24) implies $F$ is bounded from below. Thus any convergent minimizing sequence in
$W_M^{\ast}$ cannot tend to $-\infty$.

\section{Proof of Theorem \ref{maintheorem}}.

In this section we prove Theorem \ref{maintheorem}.  Statement (1) in  Theorem \ref{maintheorem} is crucial. With the aid of (1) in Theorem \ref{maintheorem}, (2) and (3) can be proved as in \cite{44} and \cite{45}. Therefore, the key is to prove (1) which is given by two lemmas.

First, we  prove that the functional $F(\rho)$ is bounded below on the set $W_M^{\ast}$ if $\gamma>2$.

\begin{lem}\label{lem4.2} Suppose  that $\gamma>2$.   Then

\be\label{boundedbelow'}
F(\rho)\ge C_1\int_{\mathbb{R}^3}\rho^{\gamma}d_3x-C_2, \qquad \rho\in W_M^{\ast},
\ee
for some positive constants constants $C_1$ and $C_2$ independent of $\rho$.\end{lem}
\begin{proof}
For $\rho\in W_M^{\ast}$, Let
\be\label{eq1} F(\rho)=\tilde F(\rho)+\int_{\mathbb{R}^3} \rho(x)\mathfrak{P}(\r)(x)d_3x. \ee
For  simplicity of  presentation, we set
\be\label{4pibeta} 4\pi\beta=-1.\ee
Let $\psi=\mathfrak{P}(\r)$; then $\psi$ satisfies the following equation
\be\label{eq3} \psi_{rr}-\frac{1}{r}\psi_r+\psi_{zz}= r^2\rho.\ee
It was shown in \cite{44} or \cite{45}, for $\rho\in W_M$ and $\gamma>4/3$, $\tilde F$ satisfies the inequality
\be\label{eq4} \tilde F(\rho)\ge c_1\int (\rho(x))^{\gamma}d_3x-c_2, \ee
for some positive constants $c_1$ and $c_2$ independent of $\rho$.
The main task is to estimate the term
\be\label{eq5} Q=\int\rho(x)\psi(x)d_3x.\ee
To this end, we make the change of variable
\be\label{eq6} \psi(r,\ z)=r^a\chi (r,\ z), \ee
where $a$ is a constant to be determined. We compute
$$\psi_{rr}-\frac{1}{r}\psi_r=r^a\chi_{rr}+(2a-1)r^{a-1}\chi_r+a(a-2)r^{a-2}\chi. $$
Taking $a=2$ gives
$$\psi_{rr}-\frac{1}{r}\psi_r=r^2\chi_{rr}+3r\chi_r,$$
so using \eqref{eq3} we get
\be\label{eq7}
\chi_{rr}+\frac{3}{r}\chi_r+\chi_{zz}=\rho.\ee
Noting that we are working with axi-symmetric functions, we recognize the left side of \eqref{eq7} to be related to the Laplacian of $\chi$ in 5-dimensions.
To make this precise, we must first extend our functions $\rho$ and $\chi$ from $\mathbb{R}^3$ to $\mathbb{R}^5$. We do this as follows: \\
Let $x_3=z$, $r=\sqrt{x_1^2+x_2^2}$; then \eqref{eq7} becomes
\be\label{eq8}
\chi_{rr}+\frac{3}{r}\chi_r+\chi_{x_3x_3}=\rho(r, x_3), \qquad \chi=\chi(r, x_3). \ee
Now write
$$\rho(x_1, x_2, x_3)=f( (x_1^2+x_2^2)^{1/2}, x_3)=f( r, x_3),$$
and define the extension of $\rho$ to $\mathbb{R}^5$ by
$$\rho_e(x_1, x_2, x_3, x_4, x_5)=f( (x_1^2+x_2^2+x_4^2+x_5^2)^{1/2}, x_3)=f(R, x_3), $$
where
\be\label{eq9}
R=(x_1^2+x_2^2+x_4^2+x_5^2)^{1/2}.\ee
Similarly, writing
$$\chi(x_1, x_2, x_3)=g((x_1^2+x_2^2)^{1/2},x_3)=g(r, x_3), $$
we extend $\chi$ to $\mathbb{R}^5$ by defining
$$\chi_e(x_1, x_2, x_3, x_4, x_5)=g( (x_1^2+x_2^2+x_4^2+x_5^2)^{1/2}, x_3)=g( R, x_3). $$
Since \eqref{eq7} can be written as
$$g_{rr}+\frac{3}{r}g_r+g_{x_3x_3}=f(r, x_3),$$
it follows that
\be\label{eq10}
\frac{\partial^2 g}{\partial R^2}+\frac{3}{R}\frac{\partial g}{\partial R}+\frac{\partial^2 g}{\partial x_3^2}=f(R, x_3)
\ee
because the functions $g(r, x_3)$ and $g(R, x_3)$ are the same functions with different names for the first variable and $f(r, x_3)$ is the same as $f(R, x_3)$, again with different names for the first variable. Thus, \eqref{eq7} gives
$$(\chi_e)_{RR}+\frac{3}{R}(\chi_e)_R+(\chi_e)_{x_3x_3}=\rho_e(R, x_3), $$
where $R$ is given in \eqref{eq9}. That is, the extended functions $\chi_e$ and $\rho_e$ satisfy
\be\label{eq11}
\Delta_5\chi_e=\rho_e, \ee
where $\Delta_5$ denotes the Laplacian in $\mathbb{R}^5$. Now it is well-known (\cite{gilbargtrudinger}) that the Green's function for $\Delta_5$ is
\be\label{eq12} G_5(x-y)=-\frac{1}{15\omega_5}|x-y|^{-3}, \ee
where $\omega_5$ is the volume of the unit 5-ball. 
Thus, from \eqref{eq11} we obtain
\be\label{eq13}
\chi_e(x)=\int_{\mathbb{R}^5}G_5(x-y)\rho_e(y)d_5y=(G_5\ast\rho_e)(x), \ee
where $\ast$ denotes the convolution operator.

\vskip 0.2cm
We shall use \eqref{eq13} to study $Q$; cf \eqref{eq5}.
Thus
\be\label{eq14}
Q=\int_{\mathbb{R}^3}\rho\psi d_3x=\int_{\mathbb{R}^3}\rho r^2\chi d_3x=K\int_{\mathbb{R}^5}\rho_e \chi_e d_5x=K\int_{\mathbb{R}^5}\rho_e(G_5\ast \rho_e) d_5x, \ee
where $K$ is the area of the unit 1-sphere divided by the area of the unit 3-sphere. Using H${\rm \ddot{o}}$lder's inequality we obtain
\be\label{eq15}
|\int_{\mathbb{R}^5}\rho_e(G_5\ast \rho_e) d_5x|\le \|\rho_e\|_s\|G_5\ast \rho_e\|_t,
\ee
where
\be\label{eq16}
\frac{1}{s}+\frac{1}{t}=1.
\ee
We would like to use Young's inequality (\cite{lieb}, P.19)
\be\label{eq17}
\|G_5\ast \rho_e\|_t=\tilde C \||x|^{-3}\ast \rho_e\|_t\le C\tilde C  \||x|^{-3}\|_q\|\rho_e\|_s,\ee
where $C=C(q, s, t)$, and
\be\label{eq18}
1+\frac{1}{t}=\frac{1}{q}+\frac{1}{s}.\ee
  To this end, we define the radial cut-off function $\delta$ by
\be\label{eq20} \delta(x-y)=\begin{cases} &1, \ {\rm if~} |rad(x-y)|\ge 2\mathfrak{R}, \\
                                           & 0, \ {\rm otherwise~}.
                                           \end{cases}\ee
Here $\mathfrak{R}$ is as in (3.14) and $|rad(x-y)|$ is the distance of $(x-y)$ from the $x_3$ (or $z$)-axis. We now note that for $\rho\in W_M^{\ast}$, we may replace $G_5$ by $\delta G_5$ in \eqref{eq14}, \eqref{eq15} and \eqref{eq17} and thus we need to study (from \eqref{eq17}),
\be\label{eq21} \|\delta(x)|x|^{-3}\|_q, \ee
where
\be\label{eq22} \|\delta(x)|x|^{-3}\|_q^q=\int_{\mathbb{R}^5}\delta(x)|x|^{-3q}d_5x. \ee
For this integral to be finite near $x=0$, we need
\be\label{eq23}q<5/3.\ee
From \eqref{eq16}, \eqref{eq18} and \eqref{eq13}, we obtain
\be\label{eq24} s>\frac{10}{7}. \ee
We will require $\gamma\ge s$. This is ensured for $\gamma>2$.
We next study the integral \eqref{eq22} at infinity.
\vskip 0.2cm
We decompose the 5-vector $x$ into its $z$ and $\bar r$ components:
$$x=a\bar {i_z}+\bar b, $$
where $\bar {i_z}$ is the unit vector in $z$-direction and $\bar b=rad(x)$. Writing $d\Omega$ for the angular element, we have
\begin{align*}\label{25}
&\|\frac{\delta(x)}{|x|^3}\|_q^q=\int_{\mathbb{R}^5}\frac{\delta(x)}{|x|^{3q}}d_5x
=\int_{-\infty}^{+\infty}da\int d\Omega\int \frac{r^3}{(r^2+a^2)^{3q/2}}dr\\
&\le \int_{-1}^{+1}\int_0^{2\mathfrak{R}}r^3dr\int d\Omega\frac{1}{(r^2+a^2)^{3q/2}}da
+2 \int_{1}^{+\infty}\int_0^{2\mathfrak{R}}r^3dr\int \frac{d\Omega}{a^{3q}}da.
\end{align*}
We will require
$$q\ge 1, $$
so both these expressions are finite . We now note from \eqref{eq14},
\be\label{eq26}
|Q|=K\int_{\mathbb{R}^5}\rho_e(G_5\ast\rho_e)d_5x\le CK\|\delta(x)|x|^{-3}\|_q\|\rho_e\|_s^2\le K C' \|\rho_e\|_s^2, \ee
where $C'=C\tilde C \|\delta(x)|x|^{-3}\|_q$ is a constant independent of $\rho_e$. This inequality, together with \eqref{eq1}, \eqref{eq4} and \eqref{eq14}, implies
\be\label{eq27}
F(\rho)\ge c_1\int_{\mathbb{R}^3} \rho^{\gamma}(x)d_3x-C-C'K\|\rho_e\|_s^2.\ee

We next estimate $\|\rho_e\|_s^2$. Before proceeding, we note that
\be\label{eq28} \int_{\mathbb{R}^5}\rho_ed_5x=C_1\int_{\mathbb{R}^3}r^2\rho d_3x\le C_1\mathfrak{R}^2M. \ee
Then
\begin{align*}
&\|\rho_e\|_s^s=\int_{\rho_e<1}\rho_e^sd_5x+\int_{\rho_e\ge 1}\rho_e^sd_5x\\
&\le \int_{\rho_e<1}\rho_e d_5x+\int_{\rho_e\ge 1}\rho_e^sd_5x.
\end{align*}
So
\be\label{29}
\|\rho_e\|_s^s\le C_1 M \mathfrak{R}^2+\int_{\rho_e\ge 1}\rho_e^2d_5x \le C_1 M \mathfrak{R}^2+\int_{\rho_e\ge 1}\rho_e^{\gamma}d_5x, \ee
and thus
\be\label{eq30}
|Q|\le C' \left(C_1 M \mathfrak{R}^2+\int_{\rho_e\ge 1}\rho_e^{\gamma}d_5x\right)^{2/s}.\ee
At this point, we need the following elmentary inequality
\be\label{31} (x+y)^a\le 2^a (x^a+y^a), \  {\rm for~}, x\ge 0,\ y\ge 0,\ a\ge 0.\ee
Using \eqref{31}, we get
\be\label{eq32}
\left(C_1 M \mathfrak{R}^2+\int_{\rho_e\ge 1}\rho_e^{\gamma}d_5x\right)^{2/s}\le (2C_1 M \mathfrak{R}^2)^{2/s}+(2\int_{\rho_e\ge 1}\rho_e^{\gamma}d_5x)^{2/s}.\ee
Applying the inequality (\cite{gilbargtrudinger})
\begin{equation} \label{eq430a}
\alpha\beta\le \epsilon\alpha^p+\epsilon^{-q/p}\beta^{q},\  \frac{1}{p}+\frac{1}{q}=1,\  p>1, q>1,
\end{equation}
gives for any $\epsilon>0$,
$$ (2\int_{\mathbb{R}^5}\rho_e^{\gamma}d_5x)^{2/s}\le \epsilon (\int_{\rho_e\ge 1}\rho_e^{\gamma}d_5x)^{{2p}/{s}}+\epsilon^{-q/p}2^{2q/s}.$$
Thus, assuming
\be\label{eq33}
s>2,
\ee
and choosing $p=\frac{s}{2}>1$, so $q=(1-\frac{1}{p})^{-1}=(1-\frac{2}{s})^{-1}>1, $
\eqref{eq32} implies
\begin{align*}
&\left(C_1M\mathfrak{R}^2+\int_{\rho_e\ge 1}\rho_e^{\gamma}d_5x\right)^{{2}/{s}}\\
&\le \left(2C_1M\mathfrak{R}^2\right)^{{2}/{s}}+\epsilon^{-q/p}2^{2q/s}+\epsilon\int_{\rho_e\ge 1}\rho_e^{\gamma}d_5x. \end{align*}
Then from \eqref{eq30}
$$|Q|\le C^{``}+C'\epsilon\int_{\rho_e\ge 1}\rho_e^{\gamma}d_5x,$$
where $$C^{``}=C'\left[\left(2C_1M\mathfrak{R}^2\right)^{\frac{2}{s}}+\epsilon^{-2/p}2^{2q/s}\right].$$
But
\begin{align*}
&C'\epsilon \int_{\rho_e\ge 1}\rho_e^{\gamma}d_5x\le  C'\epsilon \int_{\mathbb{R}^5}\rho_e^{\gamma}d_5x\\
&\le CC'\epsilon \int_{\mathbb{R}^3}r^2\rho^{\gamma}d_3x\le CC'\mathfrak{R}^2\epsilon \int_{\mathbb{R}^3}\rho^{\gamma}d_3x.\end{align*}
Using this in \eqref{eq27}, we obtain, by choosing $\epsilon$ sufficiently small,
\be\label{eq31}
F(\rho)\ge \frac{c_1}{2}\int_{\mathbb{R}^3}\rho^{\gamma}d_3x-C,
\ee
for some positive constants $c_1$ and $C$.
\end{proof}
\vskip 0.2cm
\begin{lem}\label{lem4.1} Suppose that $\gamma>4/3$, then
\be\label{negativeenergy}
\inf_{W_M^{\ast}} F(\rho)<0.\ee
\end{lem}
\begin{proof}
Let $\hat \rho$ be the compactly supported  solution for the non-rotating, non-magnetic star solution; cf (3.13). Then $\hat \rho\in W_M*$. Moreover, by the argument in \cite{44} or \cite{45}, we have
 \be \tilde F(\hat \rho)<0.\ee
 We use $\psi$ to denote $ -4\pi\beta\int_{\mathbb{R}^3} G(x, y) \rho(y)dy=\mathfrak{P}(\rho).$  Then
 $${\rm div} (\frac{1}{r^2}\nabla \psi)=-4\pi \beta\rho. $$
 Thus, for any $\rho\in W_M*$,  we use the notation in Lemma \ref{lem4.2}, i.e.,
\be\label{eq6'} \psi(r,\ z)=r^2\chi (r,\ z), \ee
so that
\be\label{eq7'}
\chi_{rr}+\frac{3}{r}\chi_r+\chi_{zz}=\rho.\ee
Thus, re-inserting $-4\pi \beta $, (see (4.3)),\eqref{eq7'} gives
$$(\chi_e)_{RR}+\frac{3}{R}(\chi_e)_R+(\chi_e)_{x_3x_3}=-4\pi \beta \rho_e(R, x_3), $$
or equivalently
\be\label{3'}
\Delta_5 \chi_e=-4\pi \beta \rho_e.\ee
Therefore,
\be\label{4'}
\chi_e=\int_{\mathbb{R}^5}G_5(x-y)(-4\pi \beta \rho_e)d_5x, \ee
where $G_5(x-y)$ is the Green's function of $\Delta_5$ given by (\ref{eq12}).
Notice that, for $\rho\in W_M*$,
\be\label{10'}
\beta\int_{\mathbb{R}^3} \rho \psi d_3x=\beta\int_{\mathbb{R}^5}\rho_e(x)\chi_e(x)d_5x= -4\pi \beta^2\int_{\mathbb{R}^5}\int_{\mathbb{R}^5}G_5(x-y)\rho_e(x)\rho_e(y)dxdy>0, \ee
due to \eqref{eq12}. This implies, for any for any $\rho\in W_M*$, $\beta\int \rho\psi(x)dx>0.$ In particular, 
 $I_2(\hat \rho)=:-\frac{1}{2}\beta\int_{\mathbb{R}^3}\hat \rho\mathfrak{P}(\hat \r) dx<0$. Since $F(\hat\rho)=\tilde F(\hat\rho)+I_2(\hat\rho)$,  we obtain $F(\hat \rho)<0$.  This proves \eqref{negativeenergy}.
 \end{proof}

 Lemmas \ref{lem4.2} and \ref{lem4.1} prove (1) in  Theorem \ref{maintheorem}. With this,  (2) and (3) in Theorem \ref{maintheorem} can be proved as in \cite{44} and \cite{45}.

\section{ Appendix}
\noindent{\bf Appendix A: Chandrasekhar Limit and the Case \boldmath{$\gamma=2$}}

\vskip 0.2cm  
In a recent paper \cite{das}, there has been a discussion of the " Chandrasekhar limit " for magnetic white dwarf stars. White dwarfs avoid gravitational collapse via " electron degeneracy pressure" (\cite {57}). This is a quantum mechanical effect resulting from the Pauli
Exclusion Principle; namely, since electrons are fermions, no two electrons can be in the same state, and therefore occupy a band of energy levels. Compression of the electrons increases the number of electrons in a given volume and raises the maximum energy level in the occupied band. Thus the energy of the electrons increases, resulting in a pressure against the gravitational compression of matter into smaller volumes of space. The  {\it Chandrasekhar limit } is the mass above which electron degeneracy pressure is insufficient to balance the stars own gravitational attraction.

In \cite{das}, the authors claim that " strongly magnetized white dwarfs not only can violate the Chandrasekhar mass limit significantly, but exhibit a different mass limit ". In their analysis they consider a polytropic equation of state $p=\rho^{\gamma}$ with $\gamma=2$. Thus it is of some interest to extend Theorem 3.3 to the case $\gamma=2$. \vspace{1em}

\noindent {\textbf{Theorem A1}}: Theorem 3.3 holds if $\gamma=2$ provided $|\beta|$ is sufficiently small.

\vskip 0.2cm

\noindent {\emph{Proof}}. It suffices to show that (4.1) holds if $\gamma=2$ for small $|\beta|$.

As before, we define $Q$ by
\be\label{eqq1} Q=-4\pi\beta \int_{\mathbb{R}^3} \rho \psi d_3x. \ee
For $s=2$, we have, as in (4.29),
$$\frac{|Q|}{4 \pi |\beta|}\le C'\left[C_1 M \tilde{R}^2+\int_{\rho_e\ge 1}\rho_e^2d_5x\right), $$
so
\be\label{2}\frac{|Q|}{4 \pi |\beta|}\le C' \left[C_1 M \tilde{R}^2+C\int_{\mathbb{R}^3}\rho^2d_3x\right], \ee
since
$$\int_{\rho_e\ge 1} \rho_e^2d_5x\le \int_{\mathbb{R}^5} \rho_e^2d_5x\le C \int_{\mathbb{R}^3} r^2 \rho^2d_3x \le C\mathfrak{R}^2 \int_{\mathbb{R}^3}  \rho^2d_3x.$$
Thus, 
\be\label{eqq3} |Q|\le 4\pi|\beta|(C'C_1M\mathfrak{R}^2)+4\pi|\beta|C\mathfrak{R}^2)\int_{\mathbb{R}^3}  \rho^2d_3x.\ee
But 
$$F(\rho)=\tilde F(\rho)-4\pi \beta \int_{\mathbb{R}^3}\rho(x)\psi(x)dx$$
where $$
\tilde F(\rho)\ge c_1\int_{\mathbb{R}^3}\rho^2d_3x-c_2.$$
Now choose $|\beta|$ so small that
$$4\pi |\beta|C\mathfrak{R}^2<\frac{c_1}{2}.$$
Then
$$F(\rho)\ge c_1\int_{\mathbb{R}^3}\rho^2d_3x-c_2-4\pi|\beta|(CC'M\mathfrak{R}^2)-\frac{ c_1}{2}\int_{\mathbb{R}^3}\rho^2d_3x,
$$
which implies: 
$$F(\rho)\ge \frac{c_1}{2}\int_{\mathbb{R}^3}\rho^2d_3x-c_2-4\pi|\beta|(CC'M\mathfrak{R}^2),
$$
and this is (4.1). $\Box$

 \vspace{1em}

The last result was valid for $\rho\in W_M^{\ast}:=\{\rho\in W_M: \rho(r, z)=0 \ {\rm for~} r \ge \mathfrak{R}\}.$
If we consider $\rho$ in a smaller class; namely,
\be\label{key'}
 W_M^{\ast\ast}:=\{\rho\in W_M: \rho(r, z)=0 \ {\rm for~} \sqrt{r^2+z^2} \ge \mathfrak{R}\},
\ee
we can reduce $\gamma$ below 2, with no restriction on $\beta$.

\noindent {\textbf{Theorem A2}}: If $\rho \in  W_M^{\ast\ast}$, then Theorem 3.3 holds for $\gamma>8/5.$
\vskip 0.2cm 
\noindent {\it Proof.} As in (\ref{eq14}), we have
$$Q=\int_{\mathbb{R}^3}\rho\psi d_3x=K\int_{\mathbb{R}^5}\rho_e\chi_e d_5x,$$
with
$$\chi_e(x)=-\frac{1}{15\omega_5}\int_{\mathbb{R}^5}\frac{\rho_e(y)}{|x-y|^{3}}d_{5}y
=-\frac{1}{15\omega_5}\int_{\Omega_5}\frac{\rho_e(y)}{|x-y|^{3}}d_{5}y,$$
where
\be\label{omega5}\Omega_5=:\{x\in \mathbb{R}^5: R=\sqrt{x_1^2+x_2^2+x_4^2+x_5^2}\le \mathfrak{R}, |x_3|\le \mathfrak{R}.\}\ee
By H$\ddot{\rm o}$lder's inequality, we have
\be\label{eeq1} |Q|\le K\|\rho_e\|_{2-\epsilon}\|\chi_e\|_{(2-\epsilon)/(1-\epsilon)}.\ee
By the Reisz potential estimate (cf.~\cite{gilbargtrudinger} Lemma 7.12, p.~159), we obtain
\be\label{eeq2}\|\chi_e\|_{(2-\epsilon)/(1-\epsilon)}\le C_p |\Omega_5|^{\mu-\delta}\|\rho_e\|_p, \ p>\frac{5(2-\epsilon)}{9-7\epsilon},\ee
for $\mu=\frac{2}{5}$, $\delta=\frac{1}{p}-\frac{1-\epsilon}{2-\epsilon}$.
Now H$\ddot{\rm o}$lder's inequality states, if $f\in L^q\cap L^r$ ($1\le q<p<r<\infty$), then
$$\|f\|_p\le \|f\|_q^a\|f\|_r^{1-a}, $$
for
$$a=\frac{p^{-1}-r^{-1}}{q^{-1}-r^{-1}}.$$
Taking $q=1$, $r=2-\epsilon$,
\be\label{eeqa}a=\frac{(2-\epsilon)/p-1}{1-\epsilon}<1,\ee
and using \eqref{eeq1}, \eqref{eeq2}, we obtain
\be\label{eeq3}
|Q|\le C |\Omega_5|^{\mu-\delta}\|\rho_e\|_{2-\epsilon}\|\rho_e\|_1^a\|\rho_e\|_{2-\epsilon}^{1-a}
= C |\Omega_5|^{\mu-\delta}\|\rho_e\|^{2-a}_{2-\epsilon}\|\rho_e\|_1^a, \ee
i.e.,
\be\label{eeq4}
|Q|\le C |\Omega_5|^{\mu-\delta}\left(\int_{\Omega_5}\rho_e d_5x\right)^a\left(\int_{\Omega_5}\rho_e^{2-\epsilon} d_5x\right)^{(2-a)/(2-\epsilon)}. \ee
Suppose
\be\label{eeq5}
\gamma>2-\epsilon. \ee
Then writing (\ref{eq430a}) in the form
\[
	\alpha\beta \leq \lambda\alpha^{p} + \lambda^{-q/p}\beta^{q}, \ \frac{1}{p} + \frac{1}{q} = 1, \ p > 1, \, q > 1
\]
with 
\[
	\alpha = \rho_{\epsilon}^{2 - \epsilon}, \ \beta = 1, \ p = \frac{2}{2 - \epsilon}, \ q = \frac{\gamma}{\gamma - (2 - \epsilon)},
\]
we obtain
\be\label{eeq6}\int_{\Omega_5}\rho_e^{2-\epsilon} d_5x\le \lambda \int_{\Omega_5}\rho_e^{\gamma}d_5x +\lambda^{-(2-\epsilon)/(\gamma-(2-\epsilon)}|\Omega_5|,\ee
for any positive constant $\lambda$.
Therefore, it follows from \eqref{eeq4} and \eqref{eeq5} that
\be\label{eeq7}|Q|\le C|\Omega_5|^{\mu-\delta}\left(\int_{\Omega_5}\rho_e d_5x\right)^a\left(\lambda \int_{\Omega_5}\rho_e^{\gamma}d_5x +\lambda^{-(2-\epsilon)/(\gamma-(2-\epsilon)}|\Omega_5|\right)^{(2-a)/(2-\epsilon)}.\ee
 We choose $a=\epsilon$, then by \eqref{eeqa}, we obtain
 \be\label{eeq8}
 p=\frac{2-\epsilon}{1+\epsilon-\epsilon^2}.\ee
 Since it is required that $p>\frac{5(2-\epsilon)}{9-7\epsilon}$ (see \eqref{eeq2}), for $p$ given by \eqref{eeq8}, this is equivalent to requiring
 \be\label{eeq9}\epsilon<\frac{2}{5}.\ee
 Moreover, we require $\gamma>2-\epsilon,$ (see \eqref{eeq5}). So if $\gamma>8/5$, \eqref{eeq9}
 is ensured. \\
 For $a=\epsilon$, we get from \eqref{eeq7},
 \be\label{eeq10}
 |Q|\le C(M, \mathfrak{R})\left(\lambda\int_{\mathbb{R}^3}\rho_{e}^{\gamma}d_3x
 +\lambda^{-(2-\epsilon)/(\gamma-(2-\epsilon)}\right),
 \ee
 for some constant $C(M, \mathfrak{R})$ depending on $M$ and $\mathfrak{R}$, by noting that \bigskip
\[
	\int_{\Omega_5}\rho_e d_5x=A\int_{-\mathfrak{R}}^{\mathfrak{R}}\int_0^{\mathfrak{R}}R^3\rho_e(R, z)dRdz\le A\mathfrak{R}^2\int_{\mathbb{R}^3}\rho d_3x=A\mathfrak{R}^2M,
 \]
 \[
 	\int_{\Omega_5}\rho_e^{\gamma} d_5x\le A\mathfrak{R}^2\int_{\mathbb{R}^3}\rho^{\gamma}d_3x, \bigskip
\] 
where $A$ is a universal constant. 
 By choosing $\lambda$ sufficiently small, we get \medskip
\[
 	F(\rho)\ge \frac{1}{2}\int_{\mathbb{R}^3}\frac{\rho^{\gamma}}{\gamma-1}d_3x-C(M, \mathfrak{R}), \medskip
\]
for $\rho\in W_M^{\ast\ast}$. This finishes the proof, using the same argument as in the proof of Theorem 3.1. $\Box$

\vskip 0.3cm

\noindent{\bf Appendix B: Non-existence of Spherically Symmetric Magnetic Stars.}

\vskip 0.2cm

Radial magnetic stars cannot exist because there are no magnetic point charges. One can see this as a consequence of $\nabla\cdot {\bf B}=0$. Namely, if ${\bf B}$ is spherically symmetric, then 
$${\bf B}=B(r)(\frac{x_1}{r},\, \frac{x_2}{r}, \, \frac{x_3}{r}), \  r=(x_1^2+x_2^2+x_3^2)^{1/2}.$$
Thus 
$$\nabla\cdot {\bf B}=\sum_{i=1}^3\partial_{x_i}(B(r)\frac{x_i}{r}), $$
and since 
 $\partial_{x_i}r=\frac{x_i}{r}$, $i=1, 2, 3$, we have
$$\partial_{x_i} (B(r)\frac{x_i}{r})=\partial_r B(r) (\frac{x_i}{r})^2+B(r)(\frac{1}{r}-\frac{x_i^2}{r^3}).$$
It follows that
$$0=\nabla\cdot {\bf B}=\partial_r B(r)+\frac{2}{r}B(r),$$
so $\partial_r (r^2 B(r))=0 $ and thus $$ B(r)=\frac{c}{r^2}, \ c=  {\rm ~const}.$$
If  $B$ is bounded as $r\to 0+$, then $B(r)=0$. \\

If we allow the singularity at $r=0$ (i.e., $c\ne 0$), then if $B_{R}$ is the $R$-ball in $\mathbb{R}^{3}$,  the magnetic energy is
$$\frac{1}{8\pi}\int_{B_R} |{\bf B}|^2d_3 x=\frac{1}{8\pi}\int_0^R\frac{c^2}{r^4}4\pi r^2dr=\frac{c^2}{2}\int_0^R\frac{dr}{r^2}=\infty. $$

\vspace{4.5em}
\noindent \underline{Acknowledgement}: JS would like to thank P.~Smereka, S.~Wu, and H.~T.~Yau for helpful conversations.

\vfill
\noindent Paul Federbush, \\
Department of Mathematics, University of Michigan, Ann Arbor\\
pfed@umich.edu \bigskip \\
Tao Luo, \\
Department of Math. \& Stat. Georgetown University, \\
tl48@georgetown.edu \bigskip \\
Joel Smoller \\
Department of Mathematics, University of Michigan, Ann Arbor\\
smoller@umich.edu

 \end{document}